\documentclass[11pt, english]{amsart}
\usepackage[utf8]{inputenc}

\usepackage{xcolor}
\usepackage{rotating}
\usepackage{pifont}
\usepackage{wasysym}
\usepackage[normalem]{ulem}

\usepackage{geometry}
\usepackage{fdsymbol}

\theoremstyle{plain}
\newtheorem{theorem}{Theorem}
\numberwithin{theorem}{section}
\newtheorem*{conjecture*}{Conjecture}
\newtheorem{prop}[theorem]{Proposition}

\newtheorem{defn}[theorem]{Definition}
\newtheorem{coro}[theorem]{Corollary}
\newtheorem{remark}[theorem]{Remark}

\def\CC{{\mathbb{C}}}

\def\PP{{\mathbb{P}}}
\def\NN{{\mathbb{N}}}
\def\QQ{{\mathbb{Q}}}

\def\cO{{\mathcal{O}}}

\def\ev{\mathrm{ev}}

\begin{document} 

%\title{A note on atoms: %\linebreak 
%irrationality without quantum computations}
\title{An atomic criterion for irrationality
\linebreak
without quantum computations}
\author[V. Benedetti, A. Fay, J. Guéré, L. Manivel, N. Perrin]{Vladimiro Benedetti, Aideen Fay, Jérémy Guéré, \linebreak  Laurent Manivel, Nicolas Perrin}
\date{\today}

\address{
Universit\'e C\^ote d'Azur, CNRS, Laboratoire J.-A. Dieudonn\'e, Parc Valrose, F-06108 Nice Cedex 2, France} 
\email{vladimiro.benedetti@univ-cotedazur.fr}

\address{
Department of Mathematics, Imperial College London, London SW7 2AZ, United
Kingdom
}
\email{aideen.fay23@imperial.ac.uk}

\address{
Univ. Grenoble Alpes, CNRS, IF, 38000 Grenoble, France}
\email{jeremy.guere@univ-grenoble-alpes.fr}

\address{
Institut de Math\'ematiques de Marseille, 
Aix-Marseille University, CNRS, I2M, UMR 7373, Marseille, France}
\email{laurent.manivel@math.cnrs.fr}

\address{
Centre de Math\'ematiques Laurent Schwartz (CMLS), CNRS, \'Ecole polytechnique,
Institut Polytechnique de Paris, F-91120 Palaiseau, France}
\email{nicolas.perrin.cmls@polytechnique.edu}

\begin{abstract}
The birational invariants introduced by Katzarkov-Kontsevich-Pantev-Yu allows one to obtain irrationality results for varieties whose quantum cohomology is well-behaved. We observe that under certain cohomological conditions, we can deduce irrationality of a very general member, from the theory of atoms without actually computing them, using only monodromy equivariance of quantum multiplication and irreducibility of the monodromy representation. Our criterion applies to the very general cubic and Gushel-Mukai fourfolds, whose irrationalities were already known, but also to the very general K\"uchle fourfold of type $(c5)$, which is a Fano manifold of index one.  
\end{abstract}

\maketitle

\section{Introduction}
The theory of \emph{Hodge atoms} developed in \cite{kkpy} is a powerful tool to address irrationality questions for smooth projective complex varieties. It was successively applied to several types of Fano fourfolds, starting from cubic hypersurfaces. 

Atoms are defined in terms of big quantum cohomology rings, for which explicit closed formulae are notoriously difficult to obtain. Small quantum cohomology can often be used as a substitute, and under favorable circumstances, provides enough information to lead to interesting conclusions about irrationality.
In particular, for Fano varieties, the small quantum cohomology only requires computing finitely many genus-zero Gromov-Witten invariants, making it more accessible.

In this short note we give a simple cohomological criterion for proving irrationality of certain very general fourfolds (see Theorem \ref{thm:irrat}). This criterion relies deeply on the existence of atoms, but interestingly it doesn't rely on any explicit quantum cohomology computations. In particular, it applies to the very general cubic fourfold (see \cite{kkpy}, \cite{jg}), and also to the very general Gushel-Mukai fourfold (see \cite{bmp}). We also apply it to a new case, that of the very general K\"uchle fourfold of type (c5); this is one of the relatively few known families of Fano fourfolds of Picard number one.

Like cubic and Gushel-Mukai fourfolds, K\"uchle fourfolds of type (c5) admit a Hodge structure of K3 type.
It is then natural to expect that if it is rational, then its vanishing cohomology is isomorphic as $\mathbb{Q}$-Hodge structures to the (twisted) middle cohomology of a projective K3 surface.
However, our criterion is not sufficient for such a strong statement and a more comprehensive study of the quantum cohomology ring is still required there.
In a parallel work \cite{Fay}, the second author has explored the exact nature of the atoms contained in a smooth K\"uchle fourfold of type (c5), proving irrationality of the very general member and going further to the result about $\mathbb{Q}$-Hodge structures stated above.

In order to use the full power of the method presented here, it would be very interesting to find an example of a fourfold where the standard techniques to compute the small quantum cohomology fail, for instance a fourfold that is a so-called \textit{non-convex} complete intersection, but for which our criterion applies. At last, we highlight that every variety in this paper is smooth, projective, and defined over the complex numbers and that it remains a big challenge to extend the theory of atoms, e.g.~to singular varieties.

\bigskip\noindent {\it Acknowledgements}. A.F. would like to thank Tom Coates for helpful discussions. V.B., L.M. and N.P. were supported by the project FanoHK ANR-20-CE40-0023, and A.F. by EPSRC grant number EP/Y028872/1.

\section{On atoms of surfaces}
\subsection{Hodge atoms}

We briefly review the theory of atoms from \cite{kkpy}, following the language of \cite{jg}.

Let $X$ be a smooth complex projective variety.
We denote by
$$\kappa^X_\tau(-):=\mathrm{Eu}^X_\tau \star^X_\tau (-)$$
the quantum multiplication by the Euler vector field at a formal Hodge class $\tau$.
The operator $\kappa_\tau$ is an endomorphism of $H^*(X,\mathbb{Q})_R:=H^*(X,\mathbb{Q})\otimes_\QQ R$ where $R$ is the graded ring $\mathbb{Q}[[Q,(T_k)]]$ of formal power series in the Novikov variable $Q$ of $X$ and the formal coordinates $(T_k)$ of the Hodge locus in $H^*(X,\mathbb{Q})$.
The ring $R$ embeds into a ring $\widehat{R}$, where one adds an extra formal variable $\mathfrak{q'}$; specifically $\widehat{R} = \mathbb{Q}[{\mathfrak{q}'}^{\pm 1}][[Q,(T_k)]]$. This extra variable is crucial in the proof of Iritani's blow-up formula for quantum cohomology \cite{Iritani}.

Let $F$ be the Levi-Civita field over $\mathbb{Q}$ and $S^*:=F[b^{\pm 1}]$ where the extra formal variable $b$ keeps track of the degree.
Let $K$ be a number field. The notion of a $K$-evaluation map  $\ev \colon \widehat{R}^*_K \to S^*_K$ was introduced in \cite{jg}, where the index $K$ means tensor product $\otimes_\QQ K$; it allows to obtain an endomorphism $\ev(\kappa_\tau)$ of $H^*(X,K)_{S_K}:=H^*(X,K)\otimes_K S_K$.
Since $\kappa_\tau$ is homogeneous of degree $2$, the endomorphism $b^{-2}\ev(\kappa_\tau)$ is naturally an endomorphism of $H^*(X,F_K)$.

For any $K$-evaluation map $\ev$, we denote by 
$$\mathrm{Sp}^X_{\ev} := \left\lbrace \alpha \in F_{\overline{\QQ}} ~|~~ \det \big(b^{-2}\ev(\kappa_\tau) -\alpha\big)=0 \right\rbrace \subset F_{\overline{\QQ}}$$
the set of eigenvalues; for any $\alpha \in F_{\overline{\QQ}}$, we define
$$E^X_{\ev,\alpha} := \ker \big(b^{-2}\ev(\kappa_\tau) -\alpha\big)^m \subset H^*(X,F_{\overline{\QQ}}) ~,~~m>>1.$$
This is the generalized eigenspace when $\alpha \in \mathrm{Sp}_\ev^X$, and zero otherwise.

\begin{defn}
Let $\ev \colon \widehat{R}^*_K \to S^*_K$ be a $K$-evaluation map and $\alpha \in \mathrm{Sp}_\ev^X$ an eigenvalue. The generalized eigenspace $E^X_{\ev,\alpha}$ is called a \emph{coarse atom} of $X$.
\end{defn}

From now on $K$ will be a number field containing the imaginary unit.

\subsection{Consequences of the Noether inequality}  

\begin{prop}
Let $\Sigma$ be a minimal surface with irregularity $q(\Sigma):=h^{1,0}(\Sigma)=0$ and geometric genus $p_g(\Sigma):=h^{2,0}(\Sigma)$. Then $h^{1,1}(\Sigma) \leq 10(1+p_g(\Sigma))$.
\end{prop}

\begin{proof}
From Riemann-Roch, we get $12\chi(\cO_\Sigma) = c_1(\Sigma)^2 + \chi_{\rm top}(\Sigma)$. Our assumptions imply that $\chi(\cO_\Sigma) = 1+p_g(\Sigma)$ and $\chi_{\rm top}(\Sigma) = 2 + 2p_g(\Sigma)+h^{1,1}(\Sigma)$, so that $h^{1,1}(\Sigma) = 10(1 + p_g(\Sigma)) - c_1(\Sigma)^2$. Since $\Sigma$ is minimal, either $c_1(\Sigma)^2 \geq 0$ or $\Sigma$ is a ruled surface (see \cite[Proposition VI.2]{beauville}). In the latter case, the condition $b_1(\Sigma)=2q(\Sigma) = 0$ implies that $\Sigma$ is in fact rational and thus again $c_1(\Sigma)^2 \geq 0$. So we can always conclude that  $h^{1,1}(\Sigma) \leq 10(1 + p_g(\Sigma))$.
\end{proof}

\begin{coro}
Let $\Sigma$ be a minimal surface with irregularity $q(\Sigma)=h^{1,0}(\Sigma) = 0$ and geometric genus $p_g(\Sigma)=h^{2,0}(\Sigma) \leq 1$. Then $b_2(\Sigma) \leq 22$, with equality if and only if $\Sigma$ is a K3 surface or an elliptic fibration over $\PP^1$.
\end{coro}

\begin{proof}
The inequality follows from the previous proposition. In case of equality, necessarily $p_g(\Sigma) = 1$ and $c_1(\Sigma)^2 = 0$. This implies that the Kodaira dimension of $\Sigma$ is either $0$ or $1$. In the first case, $\Sigma$ is a K3 surface (see \cite[Théorème VIII.2]{beauville}). 
In the second case $\Sigma$ is an elliptic surface (see \cite[Proposition IX.2]{beauville}); the vanishing of $b_1(\Sigma)$ implies that it is an elliptic fibration over $\PP^1$ and that $h^{1,1}(\Sigma)=20$ (see \cite[6.10]{schutt-shioda}).
\end{proof}

\begin{coro}
\label{cor:codim-Hodge}
Let $\Sigma$ be a surface with irregularity $q(\Sigma)=h^{1,0}(\Sigma) = 0$  and geometric genus $p_g(\Sigma)=h^{2,0}(\Sigma)$.  
Then for any $K$-evaluation map $\ev \colon \widehat{R}^*_K \to S^*_K$ for the surface $\Sigma$, and any eigenvalue $\alpha \in \mathrm{Sp}^\Sigma_K$, the span of Hodge classes in $E^\Sigma_{\ev,\alpha}$ has codimension at most $9+12p_g(\Sigma)$.
\end{coro}

\begin{proof}
If the Kodaira dimension of $\Sigma$ is $-\infty$, then $\Sigma$ is rational and all generalized eigenspaces are generated by Hodge classes; so let us assume that the Kodaira dimension is non-negative. Then the surface $\Sigma$ has a minimal model $\Sigma'$ with nef canonical divisor, which implies that $\Sigma'$ has a unique coarse atom once given a $K$-evaluation map (see \cite[Example 15]{jg}). Moreover $h^{p,0}(\Sigma) = h^{p,0}(\Sigma')$ for all $p$ since they are birational invariants (see also \cite[Section 7.1]{debarre}).

Consider the blow-up map $\Sigma \to \Sigma'$ and choose a $K$-evaluation map $\ev'$ for $\Sigma'$.
Let $\alpha_0 \in F_{\overline{\QQ}}$ be the unique element of $\mathrm{Sp}^{\Sigma'}_{\ev'}$.
By choosing arbitrary $K$-evaluation maps for the blow-up centers (which are points) such that $\alpha_0$ is not in their spectra, and by using \cite[Corollary 37]{jg} (rephrasing Iritani's blow-up formula on quantum cohomology, see \cite{Iritani}), we obtain a $K$-evaluation map $\ev$ for the surface $\Sigma$ such that
$$E^{\Sigma}_{\ev,\alpha_0} \simeq E^{\Sigma'}_{\ev',\alpha_0}=H^*(\Sigma',\QQ)\otimes_\QQ F_{\overline{\QQ}},$$
where the isomorphism preserves Hodge classes. Thus the previous proposition implies that this generalized eigenspace has dimension at most $10(1+p_g(\Sigma))+2p_g(\Sigma)+2$.
Since it contains at least three Hodge classes, it satisfies the desired condition. The other generalized eigenspaces are coming from the blow-up centers, which are points, so they are spanned by Hodge classes.
\end{proof}

\begin{coro}
\label{cor_p_g>2}
Let $\Sigma$ be a projective surface with irregularity $q(\Sigma)=h^{1,0}(\Sigma) = 0$  and geometric genus $p_g(\Sigma)=h^{2,0}(\Sigma)>1$. Then either $\Sigma$ is an elliptic surface $S\to \PP^1$ or $\Sigma$ is of general type.

For any $K$-evaluation map $\ev \colon \widehat{R}^*_K \to S^*_K$ for the surface $\Sigma$, and any eigenvalue $\alpha \in \mathrm{Sp}^\Sigma_K$, the span of Hodge classes has codimension in $E^\Sigma_{\ev,\alpha}$ at most $8+12p_g(\Sigma)$ in the former case, and at most $13+10p_g(\Sigma)$ in the latter case.
\end{coro}

\begin{proof}
As in the previous proof, we can assume that $\Sigma$ is minimal. Since $p_g(\Sigma)>1$, the Kodaira dimension of $\Sigma$ is at least one. If the Kodaira dimension is equal to one, $\Sigma$ is an elliptic surface, and since $q(\Sigma)=0$, the base of the elliptic fibration is $\PP^1$. In this case $K_\Sigma$ is not ample; since $\Sigma$ is projective, its Picard rank is at least two, so $\Sigma$ has at least four Hodge classes, and the result follows. If the Kodaira dimension is equal to two, then $\Sigma$ is of general type and the result follows from the Noether inequality $c_1(\Sigma)^2\ge 2p_g(\Sigma)-4$, which implies that $h^{1,1}(\Sigma)\leq 14+8p_g(\Sigma)$.
\end{proof}

\section{Hodge classes of hyperplane sections}

Let $Y\subset \PP^n$ be a projective fivefold and $X$ a smooth hyperplane section. We denote by $j : X \hookrightarrow Y$ the corresponding embedding. Recall that the vanishing cohomology of $X$ is defined as the kernel of the Gysin morphism \cite[2.3.3]{voisin}:  
\begin{equation}\label{van}
H^*(X,\QQ)_{\mathrm{van}} := \ker \; \big(j_* : H^*(X,\QQ) \to H^{*+2}(Y,\QQ)\big).
\end{equation}
It is nonzero only in middle degree and there is a decomposition \cite[Proposition 2.27]{voisin}
\begin{equation}\label{dec}
H^4(X,\QQ)=H^4(X,\QQ)_{\mathrm{van}}\oplus j^*H^4(Y,\QQ).
\end{equation}
An important property is that the monodromy representation on $H^4(X,\QQ)_{\rm van}$ is irreducible \cite[Theorem 3.4]{voisin}. 
The set of Hodge classes of $X$ is 
\begin{equation}\label{hodgeclasses}
H^4(X)_{\rm Hdg} = H^{2,2}(X) \cap H^4(X,\QQ).
\end{equation}
Its orthogonal for the Poincar\'e pairing is the set $H^4(X)_{\rm tr}$ of transcendental classes. Of course Hodge classes of $Y$ restrict to Hodge classes of $X$. 

\begin{prop}
\label{prop:hdg-classes}
Suppose that $h^{3,1}(Y)< h^{3,1}(X)$ and $X$ is very general. Then its only Hodge classes are those restricted from $Y$:
\begin{equation}\label{samehodgeclasses}
H^4(X)_{\rm Hdg} = j^*H^4(Y)_{\rm Hdg}.
\end{equation}
\end{prop}

\proof  For $X$ very general, according to \cite[Theorem 4.1]{voisin2} $H^4(X)_{\rm Hdg}$ is stable under the monodromy action. Hence it cannot meet $H^4(X,\QQ)_{\rm van}$; otherwise, since the latter is irreducible under this action, it would  only contain Hodge classes. But under the hypothesis that  $h^{3,1}(Y)< h^{3,1}(X)$, the Gysin morphism cannot be injective on $H^{3,1}(X)$, which therefore contains vanishing classes that are not Hodge, a contradiction. \qed 

\begin{defn}
When  (\ref{samehodgeclasses}) does hold, we say that $X$ is {\it Hodge general}.
\end{defn}

\begin{remark}
By Proposition \ref{prop:hdg-classes}, a very general hypersurface $X \hookrightarrow Y$ satisfying $h^{3,1}(Y)< h^{3,1}(X)$ is Hodge general.
\end{remark}

\section{An irrationality criterion} 

\begin{theorem}
\label{thm:irrat}
Let $Y\subset\PP^n$ be a smooth Fano fivefold. Let $X$ be a smooth Fano hyperplane section, and suppose that 
$$b_1(X) = b_3(X) =0, \qquad  p_g(X):=h^{3,1}(X) >h^{3,1}(Y), \qquad b_4(X)_{\rm van} \geq 10+12p_g(X).$$
Then if $X$ is Hodge general, it is irrational.
\end{theorem}

\begin{proof}
Assume that $X$ is Hodge general, so that every Hodge class is an ambient class. Then it follows from the deformation-invariance of Gromov-Witten invariants that the operator $\kappa_\tau$ is monodromy-equivariant, and so is the operator $\ev(\kappa_\tau)$ for any evaluation map $\ev$.
Since $H^4(X)_{\rm van}$ is an irreducible monodromy representation, the endomorphism $\ev(\kappa_\tau)$ acts by scalar multiplication on $H^4(X)_{\rm van}$. As a consequence,
%for any evaluation map $\ev$, 
there exists a unique eigenvalue $\alpha \in \mathrm{Sp}^X_{\ev}$ such that
$$H^4(X)_{\rm van} \otimes_\CC F_{\CC} \subset E^X_{\ev,\alpha} \otimes_{\overline{\QQ}} \CC.$$

Assume moreover that $X$ is rational. Consider a weak factorization from $X$ to $\PP^4$, i.e.~a sequence
$$X=X_0 \leftrightarrow X_1 \leftrightarrow X_2 \leftrightarrow X_3 \leftrightarrow \cdots \leftrightarrow X_q=\PP^4$$
of blow-ups $\leftarrow$ or blow-downs $\rightarrow$.
We denote the blow-up centers by $Y_1, \dotsc,Y_q$.

Following \cite[Remark 42]{jg}, we pick a $K$-evaluation map $\ev_q$ for $X_q=\PP^4$, and $K$-evaluation maps $\ev'_{i,k}$ for each $Y_i$ such that the blow-up is in the direction $X_{i-1} \to X_i$, where $1 \leq i \leq q$ and $1 \leq k < \mathrm{codim}_{X_i}Y_i$, in such a way that the spectra $\mathrm{Sp}^{X_q}_{\ev_q}$, $\mathrm{Sp}^{Y_i}_{\ev'_{i,k}}$
have maximal possible cardinality (see Remark \ref{rmk_evaluations}), and that their two-by-two intersections are all empty.

Then by \cite[Corollary 37]{jg}, we obtain $K$-evaluation maps $\ev_i$ and $\ev'_{i,k}$ for the remaining $X_i$ and $Y_i$.
In particular, we get a $K$-evaluation map $\ev_0$ for $X$.
By \cite[Corollary 37]{jg},  we deduce for any $\alpha \in F_{\overline{\QQ}}$ that
$$E^{\PP^4}_{\ev_q,\alpha} \oplus \bigoplus_{(i,k) \in I_{\rightarrow}} E^{Y_i}_{\ev'_{i,k},\alpha} \simeq E^{X}_{\ev_0,\alpha} \oplus \bigoplus_{(i,k) \in I_{\leftarrow}} E^{Y_i}_{\ev'_{i,k},\alpha}$$
where the isomorphism respects Hodge classes and Hochschild degrees.
Here, we introduced the subset  $$I_{\leftarrow} := \left\lbrace (i,k),  \; 1 \leq i \leq q, \; k \in \NN ~|~~X_{i-1} \leftarrow X_i \textrm{ and } 1 \leq k < \mathrm{codim}_{X_{i-1}}Y_i \right\rbrace ,$$  and similarly for $I_\to$.

Note that we have \textit{a priori} no control on the $K$-evaluation maps $\ev_i$ and $\ev'_{i,k}$ obtained from \cite[Corollary 37]{jg}.
However, since we have chosen the spectra $\mathrm{Sp}^{X_q}_{\ev_q}$, $\mathrm{Sp}^{Y_i}_{\ev'_{i,k}}$
with maximal cardinality and empty two-by-two intersections, the same property must hold for the remaining spectra.
Indeed, let us proceed by decreasing induction on $i$ to prove that the following property $(\star)_i$ holds for any $i$:
\begin{itemize}
\item if $X_i  \leftarrow X_{i+1}$, then the spectra $\mathrm{Sp}^{X_i}_{\ev_i}$, $\mathrm{Sp}^{Y_{i+1}}_{\ev'_{i+1,k}}$ have maximal cardinalities and empty two-by-two intersections;
\item if $X_{i-1} \leftarrow X_i  \rightarrow X_{i+1}$, then the spectrum $\mathrm{Sp}^{X_i}_{\ev_i}$ has maximal cardinality;
\item if $X_{i-1} \rightarrow X_i  \rightarrow X_{i+1}$, then the spectra $\mathrm{Sp}^{X_i}_{\ev_i}$, $\mathrm{Sp}^{Y_i}_{\ev'_{i,k}}$ have maximal cardinalities and empty two-by-two intersections.
\end{itemize}

Let $1 \leq i <q$; assume that Property $(\star)_{i+1}$ holds, and consider Property $(\star)_i$. In the first case, if the conclusion fails, we can choose different evaluation maps for $X_i$ and 
$Y_{i+1}$ with maximal cardinality and empty two-by-two intersections, and then get a new evaluation map for $X_{i+1}$ whose associated operator has strictly increased its number of eigenvalues; but this contradicts Property $(\star)_{i+1}$. In the second case, if by contradiction $\mathrm{Sp}^{X_i}_{\ev_i}$ does not have maximal cardinality, then we can choose another evaluation map for $X_i$ such that its associated operator has strictly more eigenvalues; by \cite[Corollary 37]{jg} this gives different evaluation maps for $X_{i+1}$ and $Y_{i+1}$ (here we are assuming $X_{i+1}\leftarrow X_{i+2}$, otherwise the argument is similar and even simpler) and at least one of the corresponding spectra will have strictly bigger cardinality than the original evaluation map at step $i+1$, thus contradicting maximality of the original spectra. In the third case, the maximality of $\mathrm{Sp}^{X_i}_{\ev_i}$ follows from the same argument as in the previous case; the maximality of $\mathrm{Sp}^{Y_i}_{\ev'_{i,k}}$ was supposed from the beginning. Finally, by applying  \cite[Corollary 37]{jg} repeatedly, we get that $\mathrm{Sp}^{X_i}_{\ev_i}$ is contained in the union of the spectra $\mathrm{Sp}^{X_q}_{\ev_q}$ and $\mathrm{Sp}^{Y_j}_{\ev'_{j,k}}$ for $j>i$ and the blow-up map in the direction $X_{j-1}\rightarrow X_{j}$; again,  by our initial choice of evaluation maps, this implies that  $\mathrm{Sp}^{X_i}_{\ev_i}$ and $\mathrm{Sp}^{Y_i}_{\ev'_{i,k}}$ have empty two-by-two intersections.

Now, recall that the cohomology of Hochschild degree $k$ is %defined as
$$H^{(k)}:=\bigoplus_{p-q=k}H^{p,q}.$$
Let $\alpha_0 \in \mathrm{Sp}^X_{\ev_0}$ be the unique eigenvalue such that
$$H^4(X)_{\rm van} \otimes_\CC F_{\CC} \subset E^X_{\ev_0,\alpha_0} \otimes_{\overline{\QQ}} \CC.$$
Note in particular that $E^X_{\ev_0,\alpha_0}$ contains cohomology classes of Hochschild degree $2$ since $p_g(X)>h^{3,1}(Y)$. 
Since we declared above that the two-by-two intersections of spectra $\mathrm{Sp}^{X_q}_{\ev_q}$, $(\mathrm{Sp}^{Y_i}_{\ev'_{i,k}})_{i \in I_\to}$ are all empty, since $\PP^4$ has no cohomology classes of Hochschild degree $2$, and since Property $(\star)_i$ holds for each $i$,  there exists a unique element $(i_0,k_0) \in I_\to$ such that
$$E^{Y_{i_0}}_{\ev'_{i_0,k_0},\alpha_0} \simeq E^{X}_{\ev_0,\alpha_0}.$$
Moreover, the blow-up center $Y_{i_0}$ must be a surface, as it contains a cohomology class of Hochschild degree $2$, and its minimal model $\Sigma$ has $K_\Sigma$ nef.
We also have
$$E^{Y_{i_0}}_{\ev'_{i_0,k_0},\alpha_0} \simeq H^*(\Sigma,\QQ) \otimes_\QQ F_{\overline{\QQ}} \oplus F_{\overline{\QQ}}^m~,~~ m \in \mathbb{N},$$
as the surface $Y_{i_0}$ is obtained from $\Sigma$ by blowing-up points (and so, in order to obtain $E^{Y_{i_0}}_{\ev'_{i_0,k_0},\alpha_0}$, one may have to add to the coarse atom of the minimal model some - i.e. $m$ - copies of the one dimensional Hodge atom of a point), see Remark \ref{rmk_evaluations}.

Since $H^{(1)}(X)=0$, the generalized eigenspace $E^{X}_{\ev_0,\alpha_0}$ contains no cohomology classes of Hochschild degree one. Therefore $b_1(Y_{i_0}) = b_1(\Sigma)=0$, and Corollary \ref{cor:codim-Hodge} applies.
However, by the assumption of the Theorem and the inclusion $E^{X}_{\ev_0,\alpha_0} \subset H^*(X,F_{\overline{\QQ}})$, we have
$$10+12p_g(Y_{i_0}) \leq 10+12p_g(X) \leq h^4(X)_{\rm van} \leq \mathrm{dim}_{F_{\CC}}(E^{X}_{\ev_0,\alpha_0})=\mathrm{dim}_{F_{\CC}}(E^{Y_{i_0}}_{\ev'_{i_0},\alpha_0}).$$
Since we assume there are no Hodge classes in the vanishing cohomology, we deduce that the codimension of the span of Hodge classes in $E^{Y_{i_0}}_{\ev'_{i_0},\alpha_0}$ is larger than or equal to $10+12p_g(Y_{i_0})$, contradicting Corollary \ref{cor:codim-Hodge}.
This concludes the proof. 
\end{proof}

\begin{remark}
\label{rmk_evaluations}
In the proof above, there is a technical subtlety about maximality of spectra.
Indeed, evaluation maps are defined with respect to a morphism.
Considering for instance the case of a blow-up center $\iota_i:Y_i\hookrightarrow X_{i}$, the set of evaluations maps with respect to the morphism $\iota_i$ is actually smaller than the set of evaluations maps with respect to the identity morphism $\mathrm{id}_{Y_i}$.

In the previous proof, when choosing evaluation maps (for blow-up centers) whose associated spectra have maximal cardinality, it is meant to be with respect to the embedding $\iota_i$.
Therefore, these evaluation maps may not be `maximal' with respect to the identity morphism. This is the reason why later in the proof we cannot assume that the generalized eigenspace $E^{Y_{i_0}}_{\ev'_{i_0,k_0},\alpha_0}$ is isomorphic to the cohomology of the minimal surface $\Sigma$ and that we have instead
$$E^{Y_{i_0}}_{\ev'_{i_0,k_0},\alpha_0} \simeq H^*(\Sigma,\QQ) \otimes_\QQ F_{\overline{\QQ}} \oplus F_{\overline{\QQ}}^m~,~~ m \in \mathbb{N}.$$
Note that this technical point is only relevant for the blow-up centers, see the forthcoming revision of \cite{jg}.
\end{remark}

\begin{remark}
If $h^{3,1}(Y)=0$ and $X$ is rational, the geometric genus of the surface $\Sigma$ appearing in the previous proof must be $p_g(\Sigma)=h^{3,1}(X)$. Indeed, all classes in $H^{3,1}(X)$ belong to the same coarse atom, because of irreducibility of the monodromy action.
\end{remark}

\begin{remark}
One may be interested in studying rationality of Fano fourfolds arising in smooth families but not as hyperplane sections of smooth fivefolds.
Indeed, when they are embedded in a common ambient variety, we can still define their vanishing cohomology and the monodromy action on it.
The main difficulty in applying our result in this generality lies in the lack of a good understanding of the monodromy representation beyond the hypersurface case. Many Lefschetz type results have been extended to the vector bundle case, but nothing about monodromy. In our opinion, it would be very interesting to investigate this line of research.
\end{remark}

\begin{remark}
Similarly to \cite[Definition 27]{jg}, the property of evaluation maps that is relevant for the proof of Theorem \ref{thm:irrat} is that
$$\nu^{(1),X}_{\ev,\alpha} \neq 0 \qquad \textrm{or} \qquad \nu^{(0),X}_{\ev,\alpha} - \rho^X_{\ev,\alpha} \leq 9+10\nu^{(2),X}_{\ev,\alpha}$$
for all $\alpha \in \mathrm{Sp}_\ev^X$.
Here, $\nu^{(k),X}_{\ev,\alpha}$ denotes the dimension of the subspace of $E^X_{\ev, \alpha}$ consisting of cohomology classes with Hochschild degree $k$, and $\rho^X_{\ev,\alpha}$ is the dimension of the subspace of $E^X_{\ev, \alpha}$ consisting of Hodge classes.
\end{remark}

\section{Applications}

\begin{prop}
The following fourfolds satisfy the assumptions of Theorem \ref{thm:irrat}:
\begin{enumerate}
\item cubic fourfolds,
\item Gushel-Mukai fourfolds,
\item K\"uchle fourfolds of type $(c5)$.
 \end{enumerate}
 In particular, Hodge general fourfolds of the above families are irrational.
 \end{prop}

\begin{proof}
For cubic fourfolds, $Y = \PP^5$ and the computation of Hodge numbers follows from \cite[Theorem 22.1.1]{hirzebruch}. For Gushel-Mukai fourfolds, Hodge numbers are computed in \cite[Lemma 4.1]{im}. We may choose $Y$ to be a Gushel-Mukai fivefold, then $h^{3,1}(Y) = 0$ by \cite[Proposition 3.1]{debkuz}. For K\"uchle fourfolds $X$ of type (c5), we choose $Y$ to be a K\"uchle fivefold of the same type; the Hodge numbers of $X$ are computed in \cite[Theorem 4.8]{kuchle}, and those of $Y$ in \cite[Corollary 5.4]{kuznetsov}. 
\end{proof}

\begin{remark}
Under the hypothesis of Theorem \ref{thm:irrat}, and if $X$ is rational, we have seen 
that the atom that appears in the proof comes  from a surface $\Sigma$ which is blown-up in a weak factorization of a birational isomorphism between $X$ and $\PP^4$.
In order to identify this surface, or rather its minimal model, 
it is  necessary to study the quantum cohomology ring in greater detail.
This allowed to show that it must be a K3 surface when $X$ is a cubic fourfold \cite{jg} or a Gushel-Mukai fourfold \cite{bmp}. That the same conclusion does hold for Fano fourfolds of type 
%The case when $X$ is of type
$(c5)$ will be established in the forthcoming paper \cite{Fay} by the second author.
\end{remark}

Hodge atoms are not only suited to prove irrationality results, but also provide obstructions to the existence of certain birational maps; for instance we get:

\begin{coro}
A Hodge--general K\"uchle fourfold $X$ of type (c5) is not birational to any smooth Verra fourfold. 
\end{coro}
\begin{proof}
Let $V$ be a smooth Verra fourfold and consider the proof of Theorem \ref{thm:irrat}, where $\PP^4$ is replaced by $V$.
The only difference is that we get either the existence of a unique element $(i_0,k_0) \in I_\to$ such that
$$E^{Y_{i_0}}_{\ev'_{i_0,k_0},\alpha_0} \simeq E^{X}_{\ev_0,\alpha_0},$$
or we must have
$$E^V_{\ev_q,\alpha_0} \simeq E^X_{\ev_0,\alpha_0}.$$
The former case leads to the same contradiction as before.
In the latter case, the span of Hodge classes has codimension at least $b_4(X)_{\rm van}=22$ in $E^X_{\ev_0,\alpha_0}$ (see Corollary~\ref{cor:codim-Hodge}) and it has codimension at most $21$ in $E^V_{\ev_q,\alpha_0}$ (see \cite[Section~2.1]{bgmp}); hence a contradiction.
\end{proof}

\end{document}